\documentclass[a4paper,12pt]{article}

\usepackage[centertags]{amsmath}
\usepackage{amsfonts}
\usepackage{amssymb}
\usepackage{amsthm}
\usepackage{dsfont}
\usepackage{tikz, subfigure}
\usepackage{verbatim}
\usepackage{graphicx}

\newtheorem{example}{Example}[section]
\newtheorem{theorem}{Theorem}[section]
\newtheorem{lemma}{Lemma}[section]

\DeclareMathOperator{\var}{var}

\begin{document}
\title{Bernoulli Convolutions and 1D Dynamics}
\author{Tom Kempton \& Tomas Persson}
\maketitle

\begin{abstract}{\noindent We describe a family $\phi_{\lambda}$ of dynamical systems on the unit interval which preserve Bernoulli convolutions. We show that if there are parameter ranges for which these systems are piecewise convex, then the corresponding Bernoulli convolution will be absolutely continuous with bounded density. We study the systems $\phi_{\lambda}$ and give some numerical evidence to suggest values of $\lambda$ for which $\phi_{\lambda}$ may be piecewise convex.
}
\end{abstract}

\section{Introduction}
In the study of self similar measures corresponding to non-overlapping iterated function systems, there is a natural way of defining an expanding dynamical system which preserves the measure and which allows one to study various properties such as dimension. The case of self similar measures with overlaps is much more involved, and it is not clear how best to study them using dynamical systems. 

Bernoulli convolutions are a particularly well studied family of self-similar measures. For each $\lambda\in(0,1)$ we define the corresponding Bernoulli convolution $\nu_{\lambda}$ to be the distribution of the series
\[
(\lambda^{-1}-1)\sum_{i=1}^{\infty} a_i\lambda^i
\]
where the digits $a_i$ are picked independently from digit set $\{0,1\}$ with probability $\frac{1}{2}$. Equivalently, Bernoulli convolutions are the unique probability measures satisfying the self similarity relation
\[
\nu_{\lambda}=\frac{1}{2}(\nu_{\lambda}\circ T_0+\nu_{\lambda}\circ T_1),
\]
where the maps $T_i:\mathbb R\to\mathbb R$ are defined by $T_i(x)=\frac{x}{\lambda}-(\lambda^{-1}-1)i$. For $\lambda\in(0,\frac{1}{2})$, the self similar measures are generated by a non-overlapping iterated function system, and are invariant under the interval maps $\phi_{\lambda}$ given by
\[
\phi_{\lambda}(x)=\left\lbrace\begin{array}{cc} \lambda^{-1}x &x\in[0,\lambda]\\
                         1 & x\in(\lambda,1-\lambda)\\
1-\lambda^{-1}(x-(1-\lambda))& x\in (1-\lambda,1)\end{array}\right..
\]

\begin{figure}[ht]
\hspace{\stretch{1}}
\begin{tikzpicture}[scale=3]
\draw(0,0)node[below]{\scriptsize 0}--(1,0)node[below]{\scriptsize $1$}--(1,1)--(0,1)--(0,0);
\draw[thick](0,0)--(.2,1)--(0.8,1)--(1,0);
\end{tikzpicture}
\hspace{\stretch{1}}
\begin{tikzpicture}[scale=3]
\draw(0,0)node[below]{\scriptsize 0}--(1,0)node[below]{\scriptsize $1$}--(1,1)--(0,1)--(0,0);
\draw[thick](0,0)--(.4,1)--(0.6,1)--(1,0);] 
\end{tikzpicture}
\hspace{\stretch{1}}
\begin{tikzpicture}[scale=3]
\draw(0,0)node[below]{\scriptsize 0}--(1,0)node[below]{\scriptsize $1$}--(1,1)--(0,1)--(0,0);
\draw[thick](0,0)--(.5,1)--(1,0);] 
\end{tikzpicture}
\hspace{\stretch{1}}
\caption{The maps $\phi_{\lambda}$ for $\lambda$ equal to $0.2,$ $0.4$ and $0.5$ respectively}
\end{figure}
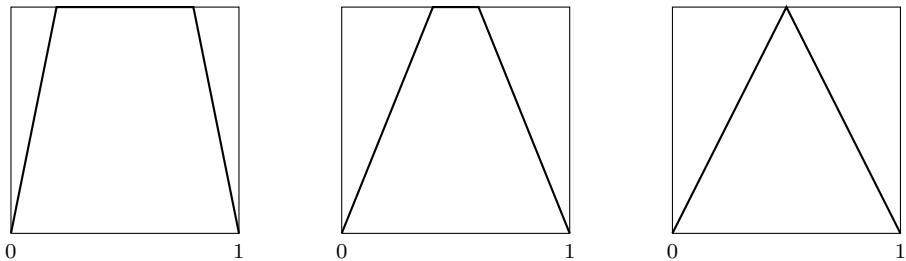

The main aim of this article is to extend the definition of $\phi_{\lambda}$ to the overlapping case, when $\lambda\in(\frac{1}{2},1)$, and to study $\nu_{\lambda}$ using these interval maps. 

There are a number of long standing open questions about Bernoulli convolutions, chief among which is the question of for which parameters $\lambda$ the corresponding measure $\nu_{\lambda}$ is absolutely continuous. It is known that each Bernoulli convolution is either purely singular or absolutely continuous, see \cite{JessenWintner}. If $\lambda$ is the inverse of a Pisot number then $\nu_{\lambda}$ is singular, see \cite{ErdosPisot}, and in fact has Hausdorff dimension less than one, \cite{Lalley}. In \cite{GarsiaAC} Garsia gave a small, explicitly defined class of algebraic integers for which $\nu_{\lambda}$ is known to be absolutely continuous. Solomyak proved in \cite{SolomyakAC} that $\nu_{\lambda}$ is absolutely continuous for almost every $\lambda\in(\frac{1}{2},1)$, and in \cite{ShmerkinAC} Shmerkin proved that the set of $\lambda\in(\frac{1}{2},1)$ admitting singular Bernoulli convolutions has Hausdorff dimension $0$, but the question of determining the parameters $\lambda$ which admit absolutely continuous Bernoulli convolutions remains open. 

There are other interesting open questions regarding Bernoulli convolutions. For example, is it the case that any singular Bernoulli convolution must have Hausdorff dimension less than one? Do there exist intervals in the parameter space for which every Bernoulli convolution is absolutely continuous (and even has continuous density)? Does the density evolve continuously with $\lambda$?

Similar questions exist in the study of invariant measures associated to various one parameter families of interval maps, and in this area a good deal of progress has been made \cite{Dobbs, FreitasTodd, Ledrappier}. With this in mind, we extend the definition of the generalised tent maps $\phi_{\lambda}$ to the overlapping case. These tent maps preserve the corresponding Bernoulli convolutions $\nu_{\lambda}$. They are described implicitly in terms of the distribution $F_{\lambda}$ of $\nu_{\lambda}$, and while we are able to write down explicit formulae for the $\phi_{\lambda}$ only in some special cases, we are able to prove some general properties. 

In particular, we prove that if $\phi_{\lambda}$ is piecewise convex for all $\lambda$ in some interval $(a,b)$ then the corresponding Bernoulli convolution is absolutely continuous with bounded density. For each $x\in[0,1]$ the map $x\mapsto\phi_{\lambda}(x)$ is continuous in $\lambda$, and convexity is preserved by passing to limits in a continuous family of functions. Thus, piecewise convexity of the functions $\phi_{\lambda}$ seems like an appropriate vehicle for passing from almost everywhere absolute continuity to everywhere absolute continuity for parameters in certain ranges. We can show that $\phi_{\lambda}$ is piecewise convex for certain special cases, and remain optimistic that one may be able to prove analytically that the map $\phi_{\lambda}$ is piecewise convex in certain parameter ranges. For the moment however, our results on piecewise convexity are restricted to some special values of $\lambda$, although we are able to run numerical approximations for any $\lambda$. There have been previous numerical investigations into Bernoulli convolutions, we mention in particular the work of Benjamini and Solomyak \cite{BenjaminiSolomyak} and of Calkin et al \cite{Calkin1, Calkin2}.

In the next section we define the maps $\phi_{\lambda}$ in which we are interested and prove that they preserve Bernoulli convolutions. We prove some elementary properties of the maps $\phi_{\lambda}$ and give the maps explicitly in some special cases. In section 3 we prove various properties of $\nu_{\lambda}$ that would follow from $\phi_{\lambda}$ being piecewise convex, and in section 4 we give some numerical evidence on the piecewise convexity of $\phi_{\lambda}$. Finally in section 5 we state some further questions and conjectures.

\section{Generalised Tent Maps}

Let $F_{\lambda}:[0,1]\to[0,1]$ be the distribution of
$\nu_{\lambda}$,
i.e.\ $F_{\lambda}(x):=\nu_{\lambda}[0,x]$. $F_{\lambda}$ is strictly
increasing because $\nu_{\lambda}$ is fully supported. We define a map
$\phi_{\lambda}:[0,1]\to [0,1]$ by
\[
\phi_{\lambda}(x) = \left\lbrace \begin{array}{cc}
  F_{\lambda}^{-1}(2F_{\lambda}(x)) & x \in
  [0,\frac{1}{2}]\\ F_{\lambda}^{-1}(2F_{\lambda} (1-x)) & x \in
  [\frac{1}{2},1] \end{array}\right. .
\]
Since $F_{\lambda}$ is strictly increasing on $[0,1]$, the map
$\phi_{\lambda}$ is well defined. We will see later that
$\phi_{\lambda}$ preserves $\nu_{\lambda}$.

\begin{figure}[ht]
\centerline{\includegraphics[scale=0.80]{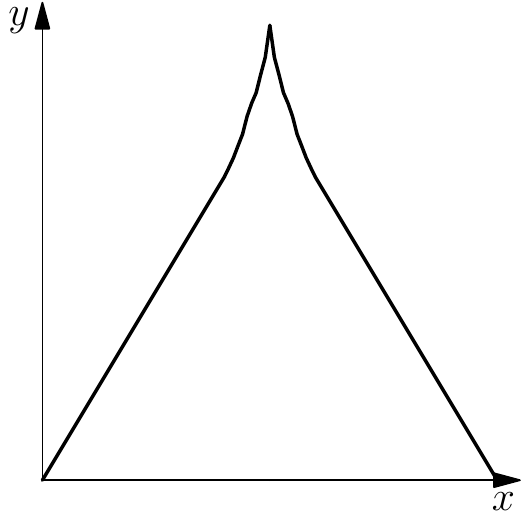} \quad
\includegraphics[scale=0.80]{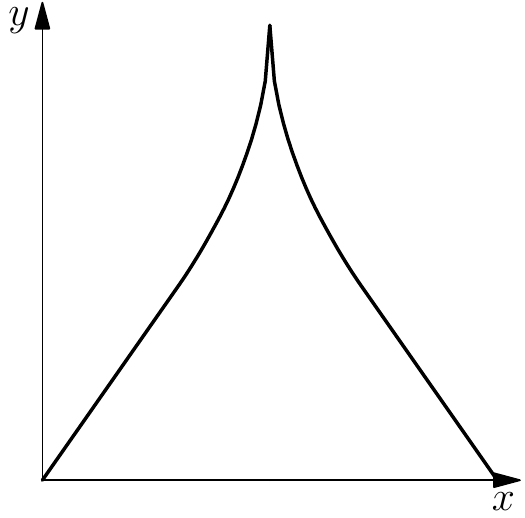} \quad
\includegraphics[scale=0.80]{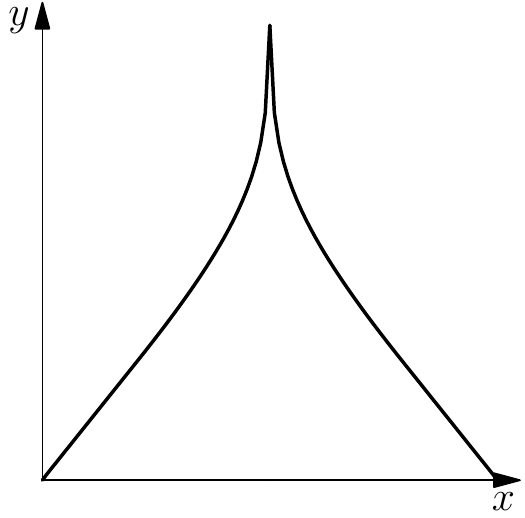}}
\caption{Graphs of $\phi_\lambda$ for $\lambda = 0.6, 0.7$ and $0.8$.}
\label{fig:tentmaps}
\end{figure}

The map $\phi_{\lambda}$ will be the chief object of study for this
article. Since $F_{\lambda}$ can be well approximated numerically, to
known levels of accuracy, one can gain good numerical approximations
to the maps $\phi_{\lambda}$. Three such approximations are displayed
for different values of $\lambda$ in Figure~\ref{fig:tentmaps}.

We begin by observing some simple properties of $\phi_{\lambda}$.

\begin{lemma}
The map $\phi_{\lambda}$ has the following properties for $\lambda\in(\frac{1}{2},1)$.
\begin{enumerate}
 \item $\phi_{\lambda}(0)=\phi_{\lambda}(1)=0$.
 \item $\phi_{\lambda}\left(\frac{1}{2}\right)=1$.
 \item $\phi_{\lambda}$ is strictly increasing on $\left[0,\frac{1}{2}\right]$ and strictly decreasing on $\left[\frac{1}{2},1\right]$.
 \item $\phi_{\lambda}(x)=\phi_{\lambda}\left(1-x\right)$.
\item $\phi_{\lambda}$ is continuous.
\end{enumerate}

\begin{proof}
We have that $F_{\lambda}(0)=0, F_{\lambda}(\frac{1}{2})=\frac{1}{2}$ and $F_{\lambda}(1)=1$ because $\nu_{\lambda}$ is supported on $[0,1]$ and symmetric about the point $\frac{1}{2}$. Then points $1$ and $2$ follow immediately.

Part 4 can be seen to be true by looking at the piecewise definition of $\phi_{\lambda}$.  Because $\nu_{\lambda}[a,b]>0$ for each $0\leq a<b\leq 1$ we have that $F_{\lambda}$ is strictly increasing. Consequently $\phi_{\lambda}$ is strictly increasing on $\left[0,\frac{1}{2}\right]$ (and strictly decreasing on $[\frac{1}{2},1]$).

Finally, we observe that continuity of $\phi_{\lambda}$ follows from the fact that $\nu_{\lambda}$ is non-atomic and that $\nu_{\lambda}[a,b]>0$ for each $0\leq a<b\leq 1$. Then both $F_{\lambda}$ and $F_{\lambda}^{-1}$ are uniformly continuous, and so $\phi_{\lambda}$ is continuous in $x$.
\end{proof}

\end{lemma}
We call maps satisfying the above properties generalised tent maps. Our first theorem is the following
\begin{theorem}
Let $\lambda\in(\frac{1}{2},1)$. Then $\nu_{\lambda}$ is invariant under $\phi_{\lambda}$. 
\end{theorem}
\begin{proof}
It is enough to show that for each $a\in[0,1]$ we have that 
\[\nu_{\lambda}[a,1]=\nu_{\lambda}(\phi_{\lambda}^{-1}[a,1]).\]
To prove this, we note that 
\[\phi_{\lambda}^{-1}[a,1]=[b,1-b]\]
where $b\in[0,\frac{1}{2}]$ satisfies $\phi_{\lambda}(b)=a$. But then
\begin{eqnarray*}
\nu_{\lambda}[b,1-b]&=&2\nu_{\lambda}[b,\frac{1}{2}]\\
&=& 2(F_{\lambda}(\frac{1}{2})-F_{\lambda}(b))\\
&=& 1-2F_{\lambda}(b)\\
&=& F_{\lambda}(1)-F_{\lambda}(\phi_{\lambda}(b))\\
&=&F_{\lambda}(1)-F_{\lambda}(a)\\
&=& \nu_{\lambda}[a,1],
\end{eqnarray*}
as required.\end{proof}

Thus, if $\nu_{\lambda}$ is absolutely continuous we see that
$\phi_{\lambda}$ has an acip. We have not been able to prove the
converse statement, that $\phi_{\lambda}$ does not have an acip in the
case that $\nu_{\lambda}$ is singular, this would be a useful
statement which would make the relationship between the study of
$\phi_{\lambda}$ and the measures $\nu_{\lambda}$ a little more
straightforward.

The following theorem shows that the maps $\phi_{\lambda}$ evolve continuously in $\lambda$.

\begin{theorem}
For each $x\in[0,1],\lambda_0\in(\frac{1}{2},1)$ we have that $\phi_{\lambda}(x)\to \phi_{\lambda_0}(x)$ as $\lambda\to\lambda_0$.
\end{theorem}

\begin{proof}
Fix $\lambda_0\in(\frac{1}{2},1)$. We rely on three facts for this proof. 

Firstly we use that the function $F_{\lambda}^{-1}$ is continuous in $x$: for all $\epsilon_2>0$ there exists $\epsilon_1>0$ such that
\begin{equation}\label{l1}
|x-y|<2\epsilon_1 \implies |F_{\lambda}^{-1}(x)-F_{\lambda}^{-1}(y)|<\epsilon_2.
\end{equation}

Secondly we use that for each $x\in[0,1]$ the function $F_{\lambda}(x)$ is continuous in $\lambda$: for all $\epsilon_1>0$ there exists $\delta_1>0$ such that
\begin{equation}\label{l2}
|\lambda-\lambda_0|<\delta_1 \implies |F_{\lambda}(x)-F_{\lambda_0}(x)|<\epsilon_1.
\end{equation}
Finally we use that for each $x\in [0,1]$ the function $F_{\lambda}^{-1}(x)$ is continuous in $\lambda$. For all $\epsilon_3>0$ there exists a $\delta_2>0$ such that
\begin{equation}\label{l3}
|\lambda-\lambda_0|<\delta_2 \implies |F_{\lambda}^{-1}(x)-F_{\lambda_0}^{-1}(x)|<\epsilon_3.
\end{equation}
We fix $x$ and let $\delta=\min\{\delta_1,\delta_2\}$ and $|\lambda-\lambda_0|<\delta$. Then

\begin{eqnarray*}
|\phi_{\lambda}(x)-\phi_{\lambda_0}(x)|&=&|F_{\lambda}^{-1}(2F_{\lambda}(x))-F_{\lambda_0}^{-1}(2F_{\lambda_0}(x))|\\
&\leq& \sup_{2F_{\lambda}(x)-2\epsilon_1\leq y\leq 2F_{\lambda}(x)+2\epsilon_1}|F_{\lambda}^{-1}(2F_{\lambda}(x))-F_{\lambda_0}^{-1}(y)|\\
&\leq& |F_{\lambda}^{-1}(2F_{\lambda}(x))-F_{\lambda_0}^{-1}(2F_{\lambda}(x))|+\epsilon_2\\
&\leq& \epsilon_3+\epsilon_2,
\end{eqnarray*}
Here the second line holds since
$2F_{\lambda_0}(x)\in(2F_{\lambda}(x)-2\epsilon,2F_{\lambda}(x)+2\epsilon)$
by equation \ref{l2}. Then the third and fourth line follows from
equations \ref{l1} and \ref{l3} respectively. Since $\epsilon_2,
\epsilon_3$ were arbitrary, we are done.
\end{proof}

In the case that one knows the distribution $F_{\lambda}$, one can
write down the map $\phi_{\lambda}$ explicitly. In particular, for
$\lambda=2^{-\frac{1}{n}}$, it is not difficult to write down
$\phi_{\lambda}$. 

\begin{example}In the case $\lambda=\frac{1}{\sqrt{2}}$,
$F_{\lambda}$ is given by
\[
F_{\lambda}(x) = \left \lbrace \begin{array}{ll} (\frac{3}{4} \sqrt{2}
  + 1) x^2 & x \in \left[0,\frac{1}{1 + \sqrt{2}}\right]\\ (1 +
  \frac{1}{\sqrt{2}}) x - \frac{\sqrt{2}}{4} & x \in [\frac{1}{1 +
        \sqrt{2}}, \frac{\sqrt{2}}{1 + \sqrt{2}}]\\ 1 - (1 +
    \frac{3}{4} \sqrt{2} ) (1 - x)^2 & x \in[\frac{\sqrt{2}}{1 +
        \sqrt{2}},1] \end{array} \right. .
\]
Consequently $\phi_{\lambda}$ is given by
\[
\phi_{\lambda}(x) = \left\lbrace\begin{array}{ll} \sqrt{2}x
&x\in[0,\frac{1}{2 + \sqrt{2}}]\\ (1 + \sqrt{2}) x^2 + \frac{1}{2 + 2
  \sqrt{2}} & x \in \left[ \frac{1}{2 + \sqrt{2}}, \frac{1}{1 +
    \sqrt{2}} \right]\\ 1 - 2 \left( \frac{1/2 - x}{1 + \sqrt{2}}
\right)^{1/2} & x \in\left[\frac{1}{1+\sqrt{2}}, \frac{1}{2}
  \right] \end{array} \right.
\]
which is extended to the whole interval $I_{\lambda}$ using the
symmetry around $\frac{1}{2}$. We have drawn the graphs of
$F_\lambda$ and $\phi_\lambda$ in Figure~\ref{fig:example}.

\begin{figure}
\centerline{\includegraphics[scale=0.8]{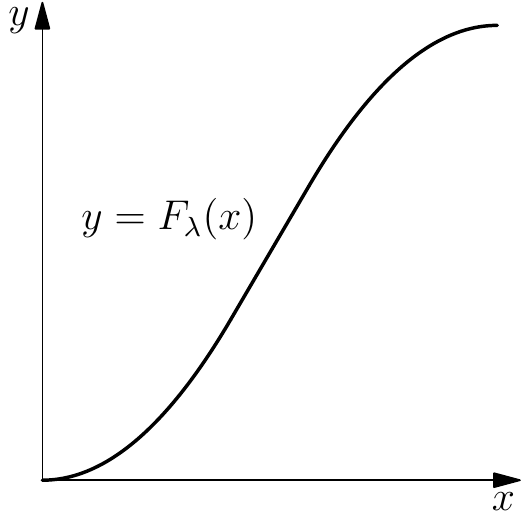} \qquad \includegraphics[scale=0.8]{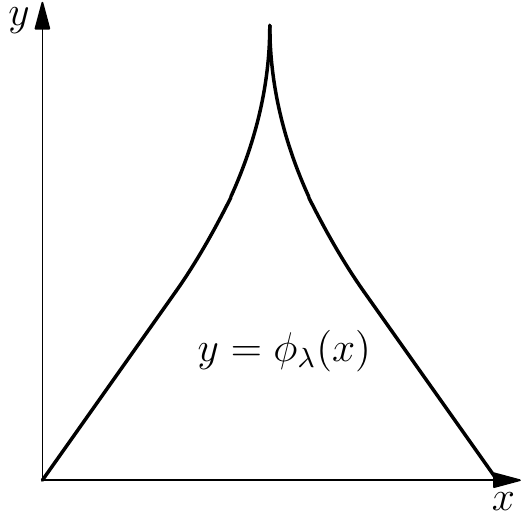}}
\caption{Graphs of $F_\lambda$ and $\phi_\lambda$ for $\lambda =
  1/\sqrt{2}$.}
\label{fig:example}
\end{figure}

\end{example}


\subsection{Further properties of $\phi_{\lambda}$}
While we cannot write down $\phi_{\lambda}$ explicitly, we can describe the behaviour near $x=0$ and the rate of the blowup at $x=\frac{1}{2}$. The following lemma describes $\phi_{\lambda}$ near $0$, and hence also the behaviour near $1$.

\begin{lemma} \label{lem:linear}
We have that 
\[
\phi_{\lambda}(x)=\lambda^{-1} x
\] 
for $x\in[0,1-\lambda]$.
\end{lemma}
\begin{proof}
Self similarity of the measures $\nu_{\lambda}$ give that
\begin{equation}\label{SelfSim}
F_{\lambda}(x)=\frac{1}{2}\left(F_{\lambda}\left(\lambda^{-1} x\right)+F_{\lambda}\left(\lambda^{-1} x-(\lambda^{-1}-1)\right)\right)
\end{equation}
Then
\[F_{\lambda}(\phi_{\lambda}(x))=2F_{\lambda}(x)=F_{\lambda}(\lambda^{-1} x)+ F_{\lambda}(\lambda^{-1} x-(\lambda^{-1}-1)).\] 
But because $F_{\lambda}(x)=0$ for $x\leq 0$, we have \[F_{\lambda}(\lambda^{-1} x-(\lambda^{-1}-1))=0\] for $x\leq 1-\lambda$. Then
\[
F_{\lambda}(\phi_{\lambda}(x))=2F_{\lambda}(x)=F_{\lambda}(\lambda^{-1} x),
\]
for $x\in[0, 1-\lambda]$, which completes the proof.
\end{proof}

It remains to find $\phi_{\lambda}(x)$ for
$x\in\left[1-\lambda,\frac{1}{2}\right]$, and then by symmetry to
define $\phi_{\lambda}$ on $[\frac{1}{2},1]$. We can also describe the
nature of $\phi_{\lambda}$ around $x=\frac{1}{2}$ for typical
$\lambda$.

\begin{lemma} \label{lem:blowup}
  We have that \[\phi \left(\frac{1}{2} - x\right) \approx 1 - c x^{-
    \frac{\log \lambda}{\log 2}}\] for small $x$, where $c$ is a
  constant that depend continuously on $\lambda$.
\end{lemma}

\begin{proof}
  We start by noting that, since $\phi_{\lambda}(x)$ evolves
  continuously in $\lambda$, it is enough to describe the nature of
  the blowup for values of $\lambda$ corresponding to absolutely
  continuous $\nu_{\lambda}$, as by passing to limits we get the
  result for all $\lambda$.
  
  We consider the behaviour of $F_{\lambda}(x)$ close to
  $x=\frac{1}{2}$ and to $x=1$. Assuming that
  $h_{\lambda}(\frac{1}{2})$ exists and is positive, we have that
  \[
  F_{\lambda}\left(\frac{1}{2}-\epsilon\right) \approx F_{\lambda}\left(\frac{1}{2}\right) -
  h_{\lambda}\epsilon = \frac{1}{2} - h_{\lambda}\epsilon.
  \]
  Thus we have that
  \begin{equation}\label{A1}
    \left|\frac{1}{2} - F_{\lambda} (\frac{1}{2} - \frac{1}{2}\epsilon)\right|
    \approx \frac{1}{2}\left|\frac{1}{2} - F_{\lambda} (\frac{1}{2} -
    \epsilon)\right|.
  \end{equation}
  Conversely, equation \ref{SelfSim} gives that for small $\delta$
  \begin{eqnarray*}
    F_{\lambda}(1-\delta) &=& \frac{1}{2} (1 +
    F_{\lambda}(\lambda^{-1}(1-\delta) - (\lambda^{-1} - 1)))\\ &=&
    \frac{1}{2} + \frac{1}{2} F_{\lambda} (1 - \lambda^{-1}\delta)
  \end{eqnarray*}
  giving
  \begin{equation}\label{A2}
    1-F_{\lambda}(1-\delta) = \frac{1}{2}(1-
    F_{\lambda}(1-\lambda^{-1}\delta)).
  \end{equation}
  Now suppose that
  $\phi_{\lambda}\left(\frac{1}{2}-\epsilon\right)=1-\delta$ for some fixed $\epsilon$ and $\delta$. Then by equations
  \ref{A1} and \ref{A2} we have that
  \[
  \phi_{\lambda}(\frac{1}{2}-\frac{1}{2}\epsilon)\approx 1-\lambda\delta.
  \]
  Iterating, we have
  \[
  \phi_{\lambda}\left(\frac{1}{2}-\left(\frac{1}{2}\right)^n\epsilon\right)\approx
  1-\lambda^n\delta,
  \]
  and we see that we have a blow up of the form
  \[
  \phi_{\lambda}\left(\frac{1}{2} - x\right) \approx 1 -
  \frac{\delta}{\varepsilon^{-\log \lambda / \log 2}} x^{- \frac{\log
      \lambda}{\log 2}}
  \]
  with $c = \delta \varepsilon^{\log \lambda / \log 2}$ depending
    continuously on $\lambda$.
\end{proof}


\section{Piecewise Convexity}
Numerical approximations of the maps $\phi_{\lambda}$ suggest that
there are ranges of $\lambda$ close to $1$ in which the maps
$\phi_{\lambda}$ are piecewise convex. For the value
$\lambda=\frac{1}{\sqrt{2}}$, one can see directly from the
calculation of the previous section that $\phi_{\lambda}$ is piecewise
convex.

In this section we prove various properties of $\nu_{\lambda}$ that
would follow from $\phi_{\lambda}|_{[0,\frac{1}{2}]}$ being convex. We
use the term `piecewise convex' as shorthand for the statement that
$\phi_{\lambda}$ is convex on each of the two intervals
$[0,\frac{1}{2}]$ and $[\frac{1}{2},1]$. The following theorem shows
the relevance of the piecewise convexity of $\phi_{\lambda}$ to the
study of Bernoulli convolutions.

\begin{theorem}\label{MainThm}
  Suppose that there exists an interval $(a,b)\subset (\frac{1}{2},1)$
  such that $\phi_{\lambda}$ is piecewise convex for each $\lambda$ in
  $(a,b)$. Then for each $\lambda\in(a,b)$ the Bernoulli convolution
  $\nu_{\lambda}$ is absolutely continuous with bounded density.
\end{theorem}

We stress that if $\phi_{\lambda}$ is piecewise convex for almost
every $\lambda$ in $(a,b)$, then it is piecewise convex for all
$\lambda$ in $(a,b)$, since the maps $\phi_{\lambda}$ are continuous
in $\lambda$ and convexity is preserved by passing to continuous
limits.

\begin{proof}
  This theorem relies on results of Rychlik \cite{rychlik}.

  Given a function $g:[0,1]\to\mathbb R$, we define the total variation of $g$ by
  \[
  \var g :=\sup_{0=x_0<x_1<\cdots<x_n=1} \sum_{i=1}^n
  |g(x_i)-g(x_{i-1})|.
  \]

  The function $g$ is said to have bounded variation if $\var
  g<\infty$. Suppose that $T: [0,1] \to [0,1]$ is a piecewise
  continuous map, such that there exists a function $g$ of bounded
  variation satisfying $g = 1 / |T'|$ almost everywhere. We consider
  the transfer operator $L$ defined on functions of bounded variation
  by
  \[
  L f (x) = \sum_{T(y) = x} g(y) f(y).
  \]
  We put $g_n = g \cdot (g \circ T) \cdots (g \circ T^{n-1})$. Then
  \[
  L^n f (x) = \sum_{T^n (y) = x} g_n(y) f(y).
  \]

  Let $C_n$ denote the $(n-1)$th refinement of partition $\{[0,\frac{1}{2}],(\frac{1}{2},1]\}$ by $T$.
  

  In \cite{rychlik}, Rychlik proved that
  \begin{equation} \label{eq:rychlik}
    \var L^n f \leq \kappa \var f + D \lVert f \rVert_1,
  \end{equation}
  where $\kappa = \sup g_n + \max_{C_n} \var_{C_n} g_n$ and $D =
  \max_{C_n} \var_{C_n} g_n / |C_n|$.

  We can apply this to our tent maps, replacing $T$ with
  $\phi_{\lambda}$. Suppose that the tent map is convex on each of the
  intervals $[0, \frac{1}{2}]$ and $[\frac{1}{2}, 1]$. Then
  $\phi_\lambda$ is differentiable everywhere except for at most
  countably many points, and this derivative is increasing on
  $[0,\frac{1}{2})$ and on $(\frac{1}{2},1]$. So there exists a
  function $g$ which is of bounded variation, which satisfies the
  assumptions of \cite{rychlik}, and which satisfies $g =
  \frac{1}{|\phi_\lambda'|}$ almost everywhere. We have
  \[
  \sup g = g(0) = g(1) = \lambda \qquad \text{and} \qquad \var_{[0,
      \frac{1}{2}]} g = \var_{[\frac{1}{2}, 1]} g \leq \lambda,
  \]
  with equality if and only if $|\phi_\lambda ' (x)| \to \infty$ when
  $x \to \frac{1}{2}$ (which is the case by
  Lemma~\ref{lem:blowup}). From this we get that
  \[
  \sup g_n = \lambda^n \qquad \text{and} \qquad \var_{C_n} g_n \leq
  2^{n-1} \lambda^n,
  \]
  Combining this with \eqref{eq:rychlik} we get that
  \begin{equation} \label{eq:rychlik_lambda}
    \var L^n f \leq 2 \lambda^n \var f + \frac{2^{n-1} \lambda^n}{\min
      |C_n|} \lVert f \rVert_1.
  \end{equation}

  In the setting of our tent map, $C_n$ corresponds to the cylinders
  of generation $n$, and so $C_n$ depends continuously on $n$.
  In particular, for each $n\in\mathbb N$ the value of $\min |C_n|$
  corresponding to $\phi_{\lambda}$ is continuous in $\lambda$.

  By Rychlik \cite{rychlik}, there is a unique non-negative function
  $h_{\lambda}$ of bounded variation, such that $\lVert h_{\lambda}
  \rVert_1 = 1$ and $L h_{\lambda} = h_{\lambda}$. The function
  $h_{\lambda}$ is the density of the unique absolutely continuous
  invariant measure of $\phi_{\lambda}$. If we pick $n$ such that $2
  \lambda^n < 1$, then \eqref{eq:rychlik_lambda} implies that
  \[
  \var h_{\lambda} \leq 2 \lambda^n \var h_{\lambda} + \frac{2^{n - 1}
    \lambda^n}{\min |C_n|},
  \]
  giving
  \[
  \var h_{\lambda} \leq \frac{2^{n - 1} \lambda^n}{\min |C_n|}
  \frac{1}{1 - 2 \lambda^n}.
  \]
  Hence we have that
  \[
  \sup h_{\lambda} \leq 1 + \frac{2^{n - 1} \lambda^n}{\min |C_n|}
  \frac{1}{1 - 2 \lambda^n},
  \]
  and so $h_{\lambda}$ is bounded. Furthermore, since all of the
  quantities involved are continuous in $\lambda$, there is a uniform
  bound on $\sup h_{\lambda}$ across all of $(a,b)$.

  Now we recall that $h_{\lambda}$ was the density of the absolutely
  continuous invariant measure of $\phi_{\lambda}$. But for almost
  every $\lambda\in(a,b)$, the Bernoulli convolution $\nu_{\lambda}$
  is absolutely continuous and is preserved by $\phi_{\lambda}$, and
  so $h_{\lambda}$ is the density of $\nu_{\lambda}$. But now any
  weak$^*$ limit point of a family of measures which is absolutely
  continuous with uniformly bounded density must also be absolutely
  continuous with the same bound on the density. Therefore, since the
  family $\nu_{\lambda}$ evolves continuously (in the weak$^*$
  topology), we see that $\nu_{\lambda}$ is absolutely continuous for
  all $\lambda\in(a,b)$ and has bounded density $h_{\lambda}$.
\end{proof}

\section{Computational Techniques}

We have seen in the previous section that showing that
$\phi_{\lambda}$ is piecewise convex for all $\lambda$ in an interval
would have significant consequences for Bernoulli
convolutions. Analytically, we have been able to show piecewise
convexity only for some special values of $\lambda$ for which the
distribution $F_{\lambda}$ is already known. We remain optimistic that
some further progress could be made here, see the comments section. In
this section we show how numerical information on $\phi_{\lambda}$ can
show convexity up to a certain scale.

\subsection{Showing convexity up to a certain scale for fixed $\lambda$}

First we choose a natural number $M$ and let $x_i$ denote the point
$\frac{i}{M}$ for $i\in\{0,\ldots ,M\}$. We wish to show that

\begin{equation} \label{eq:convextoascale}
  \phi_{\lambda}(x_i) \leq \frac{1}{2} (\phi_{\lambda}(x_{i-1}) +
  \phi_{\lambda}(x_{i+1})),
\end{equation}

for $i < \frac{M}{2}$ and $\lambda$ in a certain parameter range. It
will then follow that
\[
\phi_{\lambda}(x_j)\leq \left(\frac{j-i}{k-i}\right)
\phi_{\lambda}(x_k) + \left(\frac{k-j}{k-i}\right) \phi_{\lambda}(x_i)
\]
for $0\leq i\leq j\leq k\leq \frac{M}{2}$. This is what we call
`convexity up to scale $\frac{1}{M}$'. It corresponds to the usual
definition of convexity restricted to the set of points $\{x_0,
\ldots, x_{\frac{M}{2}}\}$, using the fact that
\[
x_j=\left(\frac{j-i}{k-i}\right)x_k+\left(\frac{k-j}{k-i}\right)x_i.
\]

Because Lemma~\ref{lem:linear} that tells us that $\phi_\lambda
(x) = \lambda^{-1} x$ for $0 \leq x \leq 1- \lambda$, we only need to
check \eqref{eq:convextoascale} for $i$ with $1-\lambda < x_i <
\frac{M}{2}$.

For large $L$ we estimate $F_{\lambda}(x)$ by noting that
\[
F_{\lambda}(x) \leq F_{\lambda,L}^+ (x) := 2^{-L} \left|\{a_1\cdots
a_L\in\{0,1\}^L : (\lambda^{-1} - 1) \sum_{i=1}^L a_i\lambda^i \leq
x\}\right|
\]
and
\begin{align*}
  F_{\lambda}(x) &\geq F_{\lambda,L}^- (x) := 2^{-L} \left|\{a_1\cdots
  a_L\in\{0,1\}^L : (\lambda^{-1} - 1) \sum_{i=1}^L a_i\lambda^i \leq
  x - \lambda^L \}\right| \\ &= F_{\lambda, L}^+ (x - \lambda^{-L} ).
\end{align*}

Then given the values of $F_{\lambda,L}^+ (x_i)$ and $F_{\lambda,L}^-
(x_i)$ for $i\in\{1,\ldots, M\}$, we can bound $\phi_\lambda$ from
below and above by
\[
\phi_{\lambda,L}^- (x_i) \leq \phi_\lambda (x_i) \leq \phi_{\lambda,L}^+
(x_i),
\]
where $\phi_{\lambda,L}^-$ and $\phi_{\lambda,L}^+$ are defined for $0
\leq x_i \leq 1/2$ by
\begin{align*}
  \phi_{\lambda,L}^- (x_i) &= y && \text{where $y$ is the largest $y$
    such that} & F_{\lambda,L}^+ (y) &\leq 2 F_{\lambda,L}^- (x_i),
  \\ \phi_{\lambda,L}^+ (x_i) &= y && \text{where $y$ is the
    smallest $y$ such that} & F_{\lambda,L}^- (y) &\geq 2
  F_{\lambda,L}^+ (x_i).
\end{align*}

Hence, if we have 
\begin{equation} \label{eq:estimateconvexity}
  \phi_\lambda^+ (x_i) \leq \frac{1}{2} ( \phi_\lambda^- (x_{i-1}) +
  \phi_\lambda^- (x_{i+1})),
\end{equation}
for all $i$ with $1 - \lambda < x_i < 1/2$, then
\eqref{eq:convextoascale} holds for all $i < \frac{M}{2}$.

The inequalities \eqref{eq:estimateconvexity} can be checked with a
computer, at least up to errors in the floating point arithmetics. We
have written a program in C\footnote{The source code for this program
  is available on the homepage of the second author,
  http://www.maths.lth.se/matematiklth/personal/tomasp/}, that
calculates the approximations $\phi_{\lambda,L}^-$ and
$\phi_{\lambda,L}^+$ of $\phi_{\lambda,L}$. Figure~\ref{fig:plot}
shows four plots of the approximations, obtained from the mentioned
program.

\begin{figure}
\includegraphics[scale=0.8]{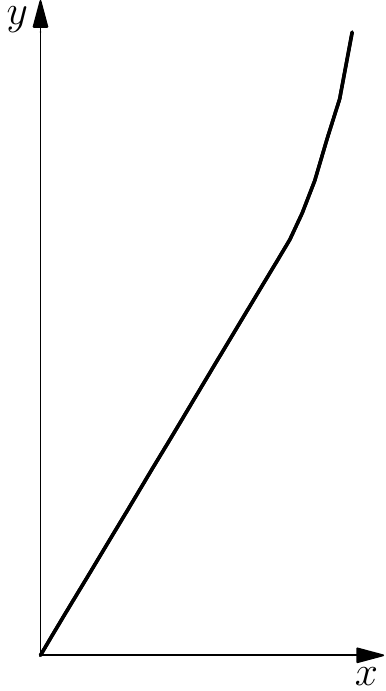} \hspace{\stretch{1}}
\includegraphics[scale=0.8]{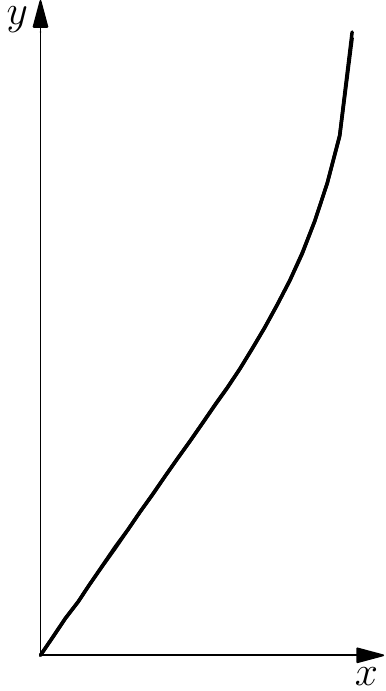} \hspace{\stretch{1}}
\includegraphics[scale=0.8]{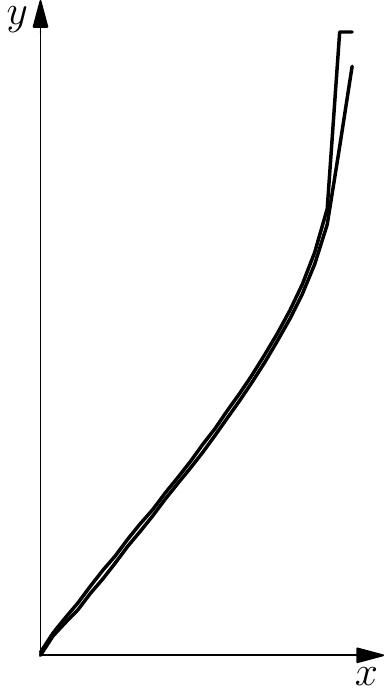} \hspace{\stretch{1}}
\includegraphics[scale=0.8]{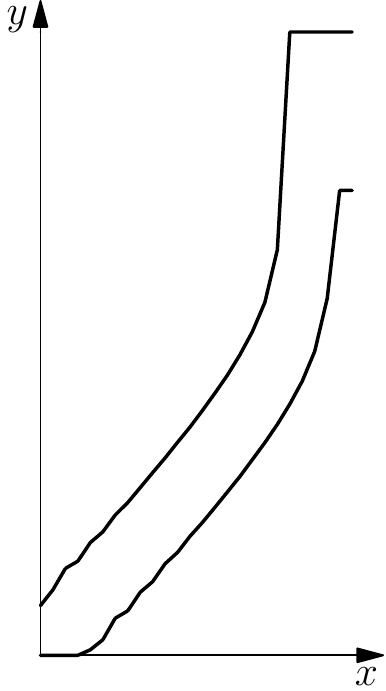}
\caption{Plot of $\phi_{\lambda,L}^-$ and $\phi_{\lambda,L}^+$ with
  $L=24$, $M=50$, and $\lambda = 0.6$ (left), $\lambda =
  0.7$, $\lambda = 0.8$ and $\lambda = 0.9$ (right).}
\label{fig:plot}
\end{figure}

\subsection{Techniques for all $\lambda$ in an interval}
To apply Theorem \ref{MainThm} we would like to show that
$\phi_{\lambda}$ is piecewise convex for all $\lambda$ in an
interval. We cannot do this computationally, but instead consider how
to show $\phi_{\lambda}$ is piecewise convex up to a certain scale for
all $\lambda$ in an interval.

We consider a small interval $I_\epsilon = [\lambda_0 - \epsilon,
\lambda_0 + \epsilon]$. The map
\[
I_\epsilon \ni \lambda \mapsto (\lambda^{-1} - 1) \sum_{i=1}^\infty
a_i \lambda^i
\]
is differentiable, and if we put $D := \frac{1}{(\lambda_0 + \epsilon)
  (1 - \lambda_0 + \epsilon)}$, then
\begin{align*}
  \biggl| \frac{\mathrm{d}}{\mathrm{d} \lambda} (\lambda^{-1} - 1)
  \sum_{i=1}^\infty a_i \lambda^i \biggr| &= \biggl|
  -\frac{1}{\lambda^2} \sum_{i=1}^\infty a_i \lambda^i + (\lambda^{-1}
  - 1) \sum_{i=1}^\infty a_i i \lambda^{i-1} \biggr| \\ &\leq \max
  \biggl\{\frac{1}{\lambda^2} \sum_{i=1}^\infty a_i \lambda^i,
  (\lambda^{-1} - 1) \sum_{i=1}^\infty a_i i \lambda^{i-1} \biggr\}
  \\ &\leq \max \biggl\{\frac{1}{\lambda^2} \sum_{i=1}^\infty
  \lambda^i, (\lambda^{-1} - 1) \sum_{i=1}^\infty i \lambda^{i-1}
  \biggr\} \\ &= \max \Bigl\{\frac{1}{\lambda (1 - \lambda)},
  \frac{1}{\lambda (1 - \lambda)} \Bigr\} = \frac{1}{\lambda (1 -
    \lambda)} \leq D,
\end{align*}
holds for all $\lambda \in I_\epsilon$ and all sequences with $a_i \in
\{0,1\}$. We conclude that
\[
(\lambda_0^{-1} - 1) \sum_{i=1}^L a_i \lambda_0^i \leq x \qquad
\Longrightarrow \qquad (\lambda^{-1} - 1) \sum_{i=1}^L a_i \lambda^i
\leq x + |\lambda - \lambda_0| D
\]
for any $\lambda \in I_\epsilon$. Similarly, we have
\[
(\lambda_0^{-1} - 1) \sum_{i=1}^L a_i \lambda_0^i \geq x \qquad
\Longrightarrow \qquad (\lambda^{-1} - 1) \sum_{i=1}^L a_i \lambda^i
\geq x - |\lambda - \lambda_0| D
\]
for any $\lambda \in I_\epsilon$. Using these two estimates we can
use $F_{\lambda_0,L}^\pm$ to estimate $F_{\lambda,L}^\pm$. We get
\[
F_{\lambda,L}^- (x) \geq F_{\lambda_0,L}^- (x - \epsilon D) \qquad
\text{and} \qquad F_{\lambda,L}^+ (x) \leq F_{\lambda_0,L}^+ (x +
\epsilon D).
\]

Hence, the estimates $F_{\lambda_0,L}^\pm$ of $F_{\lambda_0}$ gives us
estimates on $F_{\lambda,L}^\pm$ that we can use to estimate
$\phi_{\lambda,L}^-$ from below and $\phi_{\lambda,L}^+$ from above.
It is then possible to check with a computer if the inequalities in
\eqref{eq:convextoascale} are satisfied for all $\lambda \in
I_\epsilon$. This has been implemented in our program.

Table~\ref{tab:convexity} shows some result of our program. It
displays some values for which we have been able to show nummericaly
convexity to a certain scale.

\begin{table}
\begin{center}
  \begin{tabular}{l|l|c}
    $\lambda_0$ & $\epsilon$ & convexity to scale \\
    \hline
    0.65 & 0.000001 & 0.02 \\
    0.7 & 0 & 0.02 \\
    $2^{-1/2} \approx 0.707106781186548$ & 0.00001 & 0.125 \\
    0.75 & 0 & 0.02 \\
    0.75 & 0.00001 & 0.125 \\
    $2^{-1/3} \approx 0.793700525984100$ & 0.00001 & 0.125 \\
    0.8 & 0 & 0.02 \\
    0.8 & 0.00001 & 0.125 \\
    0.85 & 0.000001 & 0.125 \\
  \end{tabular}
  \caption{Numerical observations of piecewise convexity to a
    scale} \label{tab:convexity}
\end{center}
\end{table}

A convolution argument shows that $h_\lambda$ is differentiable for almost
all $\lambda \in (2^{-\frac{1}{3}}, 1)$, see \cite{SolomyakAC}. One might suspect that using
this information it would be possible to show that $\phi_\lambda$ is
piecewise convex for all $\lambda \in [2^{-\frac{1}{3}}, 1)$. However,
  this does not seem to be true, since just as we can sometimes show
  convexity to a scale using numerics, we are sometimes also capable
  of observing non-convexity at a certain scale.

Using our program we have observed that $\phi_\lambda$ is not
piecewise convex when $\lambda$ is the inverse of the root of $x^5 +
x^4 -x^2 - x -1$ that is larger than 1. We then have $\lambda \approx
0.8501\ldots$, and since $2^{-\frac{1}{3}} \approx 0.7937\ldots$, we
do not have piecewise convexity of $\phi_\lambda$ for all $\lambda \in
[2^{-\frac{1}{3}}, 1)$. In this case, $1/\lambda$ is a Salem number.

Similarly, the program can be used to show that $\phi_\lambda$ is not
convex for $\lambda = \frac{\sqrt{5}-1}{2}$. Since $\nu_\lambda$ is
known not to be absolutely continuous for this value of $\lambda$,
this is not too surprising. We also see a lack of convexity for $1/\lambda$ equal to certain other Pisot
numbers. For instance when $1/\lambda$ is the root of $x^4 - x^3 - 1$ or $x^3 - x - 1$, then $\phi_\lambda$
is not convex to scale $0.005$.


Let us mention some of the computational difficulties associated with
trying to prove convexity to a scale for the entire interval
$I_\epsilon$. Suppose $L$ is even. Our program calculates all the
$2^{L/2}$ sums $\sum_{i=1}^{L/2} a_i \lambda^i$ and stores them in an
ordered list. This requires quite a lot of memory even for $L$ as
small as $60$, but the time required to preform the calculations is
rather short. The sums in the list are then combined to get the sums
$\sum_{i=1}^{L} a_i \lambda^i$ with double as many terms, when
needed. This method of storing only the sums of length $L/2$ instead
of storing the sums of length $L$, saves memory but increases
computation time. However we found that doing so yields a better
balance between the use of memory and the computation time.

When $\lambda$ is close to $1$, large values of $L$ are needed to get
a good accuracy in the estimates, requiring an unrealistic amount of
memory. This is clearly illustrated in Figure~\ref{fig:plot}, where
for $\lambda = 0.9$, the two maps $\phi_{\lambda,L}^+$ and
$\phi_{\lambda,L}^-$ differ quite a lot, while for $\lambda = 0.6$
they are indistinguishable.

With a computer with 64 GB of memory we are able to run our program
for $L \leq 60$. Table~\ref{tab:convexity} was obtained from running
the program with $56 \leq L \leq 60$.





\section{Further Questions}
There are a number of natural questions that follow on from our work.

{\bf Question 1:} Can one show that for all $\lambda$ sufficiently
close to $1$ we have that $\nu_{\lambda}$ is absolutely continuous?
For almost all $\lambda\in (2^{-1/n},1)$ one has that the density
$h_{\lambda}$ is $(n-1)$-times differentiable, does this extra
regularity of the density give rise to extra regularity in the
functions $\phi_{\lambda}$?

{\bf Question 2:} Can one show that there is an interval $J\subset
\mathbb R$ containing $1/\sqrt{2}$ such that $\nu_{\lambda}$ is
absolutely continuous for all $\lambda\in J$. Perhaps this would
involve showing that the map $\phi_{\lambda}$ evolves smoothly in a
neighbourhood of $1/\sqrt{2}$.


{\bf Question 3:} Are there any other properties of $\nu_{\lambda}$
(besides the question of absolute continuity) which could be studied
using $\phi_{\lambda}$? In particular, can one forbid the possibility
of singular Bernoulli convolutions which have Hausdorff dimension 1 by
proving a similar result for invariant measures of $\phi_{\lambda}$?
Such results do exist in the literature for one-dimensional dynamics,
see e.g. \cite{Ledrappier} and \cite{Dobbs}, but at present are not in
such a form that they would apply to $\phi_{\lambda}$.

\section*{Acknowldgements}
Many thanks to Evgeny Verbitskiy for suggesting to us that piecewise
convexity of $\phi_{\lambda}$ would be an interesting property to
study. The first author was supported by the EPSRC, grant number EP/K029061/1.


\bibliographystyle{plain} 
\bibliography{betaref}

\end{document}